\documentclass{amsart}
\usepackage{tikz,tikz-cd}
\usepackage{url}
\usepackage{graphicx}
\usepackage{graphicx}
\usepackage{color}
\usepackage{times}
\usepackage{cite}
\usepackage{enumerate,latexsym}
\usepackage{latexsym}
\usepackage{amsmath,amssymb}
\usepackage{graphicx}
\usepackage{amsthm}
\usepackage{verbatim}
\newtheorem{thm}{Theorem}[section]
\newtheorem*{theorem*}{Theorem}
\newtheorem*{acknowledgement*}{Acknowledgement}
\newtheorem{cor}[thm]{Corollary}

\newtheorem{lem}[thm]{Lemma}
\newtheorem{prop}[thm]{Proposition}
\theoremstyle{definition}

\theoremstyle{remark}
\newtheorem{rem}[thm]{Remark}

\newtheorem{conj}[thm]{Conjecture}
\numberwithin{equation}{section}

\newcommand{\set}[1]{\left\{#1\right\}}
\newcommand{\Real}{\mathbb R}

\newcommand{\dist}[0]{\mathrm{dist}}

\title[Isoperimetric property of renormalized area]{A Sharp Isoperimetric Property of the Renormalized Area of a Minimal Surface in Hyperbolic Space}
\author{Jacob Bernstein}
\address{Department of Mathematics, Johns Hopkins University, 3400 N. Charles Street, Baltimore, MD 21218}
\email{bernstein@math.jhu.edu}
\thanks{The author was partially supported by the NSF Grant  DMS-1609340 and DMS-1904674. }

\begin{document}

\begin{abstract}
  We prove an inequality  bounding the renormalized area of a complete minimal surface in hyperbolic space in terms of the conformal length of its ideal boundary. 
\end{abstract}
\maketitle

\section{Introduction}
Consider a two-dimensional minimal surface in $\mathbb{H}^n$, $n\geq 3$, with $C^m$-regular asymptotic boundary, $m\geq 2$, on the ideal boundary, $\partial_\infty \mathbb{H}^n$, of hyperbolic space.  Following Graham and Witten \cite{RobinGraham1999} and Alexakis and Mazzeo \cite{Alexakis2010, Alexakis}, for any fixed point $p_0\in \mathbb{H}^n$, such a surface has an asymptotic expansion for its area of the form
\begin{align}\label{VolExpEqn}
	Vol_{\mathbb{H}^{n}}(\Sigma\cap \bar{B}_{R}^{\mathbb{H}^n}(p_0))=c_0 \cosh R +c_{1} +o(1), R\to \infty.
\end{align} 
See Appendix \ref{GrahawWittenApp} for how this expansion relates to the one appearing in \cite{RobinGraham1999}.   This expansion is a (Riemannian) analog of the entropy considered by Ryu and Takayanagi \cite{Ryu2006, Ryu2006a}.
 Importantly, the coefficient $c_{1}=\mathcal{A}(\Sigma)$ is independent of the point $p_0$. Indeed, this quantity is the renormalized area considered by Alexakis and Mazzeo in \cite{Alexakis2010} who show
\begin{equation}\label{AMFormulaEqn}
\mathcal{A}(\Sigma)=-2\pi \chi(\Sigma)-\frac{1}{2}\int_{\Sigma} |\mathbf{A}^{\mathbb{H}^n}_\Sigma|^2 dVol_{\Sigma}.
\end{equation}
Here $\chi(\Sigma)$ is the Euler characteristic of the surface and $\mathbf{A}^{\mathbb{H}^n}_{\Sigma}$ is the (normal bundle valued) second fundamental form of $\Sigma$ -- see Lemma \ref{AMFormualLem}.
 One also has
$$
c_0=Vol_{\partial_{\infty}\mathbb{H}^{n}}(\partial_\infty \Sigma,p_0)
$$
which is the boundary length of the asymptotic boundary inside the ideal boundary of the appropriate (depending on $p_0$) compactification of $\mathbb{H}^n$ and depends both on $\Sigma$ and $p_0$-- see Section \ref{BackgroundSec} for details.  Following \cite{bernsteinColdingMinicozziEntropy2020}, define
$$
\lambda_c[\partial_\infty \Sigma]=\sup_{p_0\in \mathbb{H}^n} Vol_{\partial_{\infty}\mathbb{H}^{n}}(\partial_\infty \Sigma,p_0)
$$
which is, essentially, the $(n-1)$-conformal volume in the sense of Li-Yau \cite{Li1982} of the embedding of the boundary curve $\partial_\infty \Sigma$ into the ideal boundary of $\mathbb{H}^n$ (see also Gromov's visual volume \cite[Section 8.2]{gromovFillingRiemannianManifolds1983}) and, unlike $c_0$, is manifestly independent of the point $p_0$.

Using ideas of Choe and Gulliver \cite{Choe1992}, we bound the renormalized area of such a minimal surface by the conformal length of its ideal boundary:
\begin{thm}\label{MainThm}
Let $\Sigma $ be  a two-dimensional minimal surface in $\mathbb{H}^n$, $n\geq 3$, with a $C^2$-regular asymptotic boundary, then 
$$
-2\pi\geq - \lambda_c[\partial_\infty\Sigma]\geq \mathcal{A}(\Sigma) 
$$
with equality throughout if and only if $\Sigma$ is a totally geodesic $\mathbb{H}^2\subset \mathbb{H}^n$.

Moreover, if $\mathcal{A}(\Sigma)= - \lambda_c[\partial_\infty\Sigma]$, then $\Sigma$ is a totally geodesic $\mathbb{H}^2$.  In particular, either $\Sigma$ is a totally geodesic $\mathbb{H}^2$ or
$$
-2\pi>-\lambda_c[\partial_\infty \Sigma]> \mathcal{A}(\Sigma).
$$
\end{thm}
\begin{rem}\label{Remarkone}
	When $\Sigma$ is a topological disk, then the absolute bound $ -2\pi\geq \mathcal{A}(\Sigma)$ and corresponding rigidity result follows from \eqref{AMFormulaEqn}. In fact, Alexakis and Mazzeo show   \cite[Section 8]{Alexakis2010} that the renormalized area of $\Sigma$ is the negative of one half of the Willmore energy of a suitably doubling of $\Sigma$.  As the Willmore energy of any closed surface is at least $4\pi$, this fact implies absolute bound in general.  Going further, the resolution of the Willmore conjecture by Marques and Neves \cite{MarquesNeves} means that if $\Sigma$ is not a disk, then $-\pi^2\geq \mathcal{A}(\Sigma)$. As such, the significance of Theorem \ref{MainThm} is the use of the boundary geometry to refine bounds on the renormalized area spectrum -- i.e., \cite[(5.15)]{Alexakis2010}.
\end{rem}

In \cite{bernsteinColdingMinicozziEntropy2020}, the author introduced a notion of entropy for submanifolds of hyperbolic space analogous to the one introduced by Colding and Minicozzi in \cite{Coldinga} for submanifolds of Euclidean space (see also \cite{zhangSuperconvexityHeatKernel2021}).  More precisely, let $H_2(t,p; t_0,p_0)$
be the heat kernel on $\mathbb{H}^2$ with singularity at $p=p_0$ at time $t=t_0$. That is, suppose it is  the unique positive solution to
$$
\left\{\begin{array}{cc} \left( \frac{\partial}{\partial t}-\Delta_{\mathbb{H}^2}\right)H_2=0 & t>0\\
\lim_{t\downarrow t_0} H_2 =\delta_{p_0}.\end{array}\right.
$$
It follows from the symmetries of $\mathbb{H}^n$ that 
$$
H_2(t, p;t_0, p_0)=K_2(t-t_0,\dist_{\mathbb{H}^2}(p,p_0))
$$
where $K_2(t,\rho)$ is a positive function on $(0,\infty)\times (0,\infty)$ and $\dist_{\mathbb{H}^2}(p,p_0)$ is the hyperbolic distance between $p$ and $p_0$. 
For $(t,p)\in (-\infty, t_0)\times \mathbb{H}^{n}$ let 
$$
\Phi_2^{t_0,p_0}(t,p)=K_2(t_0-t,\dist_{\mathbb{H}^{n}}(p,p_0)).
$$
In particular,  $\Phi_2^{t_0,p_0}$ restricts to the backwards heat kernel that becomes singular at $(t_0,p_0)$ on any totally geodesic $\mathbb{H}^2\subset \mathbb{H}^{n}$ that goes through $p_0$. 

In analogy with the Euclidean setting, for any surface $\Sigma\subset \mathbb{H}^{n}$, define the \emph{hyperbolic entropy} of $\Sigma$ to be
$$
\lambda_{\mathbb{H}}[\Sigma]=\sup_{p_0\in \mathbb{H}^{n}, \tau>0} \int_{\Sigma} \Phi_2^{0,p_0}(-\tau,p) dVol_{\Sigma}(p)
$$
This quantity is monotone non-increasing along any mean curvature flow.  Moreover, by \cite[Theroem 1.5]{bernsteinColdingMinicozziEntropy2020} if $\Sigma$ is  a two-dimensional minimal surface in $\mathbb{H}^n$, $n\geq 3$, with a $C^1$-regular asymptotic boundary,
$$
\lambda_c[\partial_\infty \Sigma]=2\pi \lambda_{\mathbb{H}}[\Sigma].
$$
Combining this with Theorem \ref{MainThm} yields:
\begin{cor}
	If $\Sigma$ is a two-dimensional minimal surface in $\mathbb{H}^n$, $n\geq 3$, with a $C^2$-regular asymptotic boundary,  then
	$$
 -2\pi \lambda_{\mathbb{H}}[\Sigma]\geq 	\mathcal{A}(\Sigma).
	$$
\end{cor}
\begin{rem}
	As the estimates of this paper do not depend in an essential way on embeddedness or interior regularity we can conclude that if $\Sigma$ is the image of branched minimal immersion that is not smoothly embedded, then
	$$
	-4\pi\geq 	\mathcal{A}(\Sigma).
	$$
	This also follows from the observations made in Remark \ref{Remarkone}.
\end{rem}

Finally, inspired by \cite{BWTopUniq},  the estimate on $\mathcal{A}(\Sigma)$ in terms of $\lambda_{\mathbb{H}} [\Sigma]$ leads one to expect the following refinement of Theorem \ref{MainThm} in $\mathbb{H}^3$:
\begin{conj}
Suppose $\Sigma, \Sigma'$ are two-dimensional minimal surfaces in $\mathbb{H}^3$ with a $C^2$-regular asymptotic boundaries.  If $\partial_\infty \Sigma=\partial_\infty \Sigma'$ and $\Sigma'$ is not a disk, then
$$
-3\pi>-\frac{(2\pi)^{3/2}}{\sqrt{e}}=-2\pi \lambda[\mathbb{S}^1]>\mathcal{A}(\Sigma).
$$
Here $\lambda[\mathbb{S}^1]\approx 1.52$ is the Colding-Minicozzi entropy of the round circle in $\Real^2$.
\end{conj} 
\begin{rem}
  As observed in Remark \ref{Remarkone} the proof of the Willmore conjecture implies $-2\pi \lambda[\mathbb{S}^1]>-\pi^2>\mathcal{A}(\Sigma')$.  Hence, the significance of the conjecture is that there is an improved bound for the surface $\Sigma$ of unspecified topology.
\end{rem}
\section{Background}\label{BackgroundSec}
We use the Poincar\'{e} ball model of hyperbolic space $\mathbb{H}^n$ in order to study the asymptotic properties of minimal surfaces in $\mathbb{H}^n$.  In particular, for any point $p_0\in \mathbb{H}^n$ one obtains a corresponding compactification of hyperbolic space.  The natural geometry on the ideal boundary in this compactification is one that is invariant under M\"obius transformations.

Recall, the Poincar\'{e} ball model of hyperbolic space, $\mathbb{H}^n$ is the open unit ball in Euclidean space 
$$
\mathbb{B}^n=B_1=\set{\mathbf{x}:|\mathbf{x}|<1}\subset \Real^{n}
$$
together with the Poincar\'{e} metric
$$
g_P=4\frac{d\mathbf{x}\otimes d\mathbf{x}}{(1-|\mathbf{x}|^2)^2}=\frac{4}{(1-|\mathbf{x}|^2)^2} g_{E}.
$$
Here $g_E$ is the Euclidean metric on $\mathbb{B}^n$.
That is, for any model of hyperbolic space, $(\mathbb{H}^n,g_{\mathbb{H}^n})$, there is an isometry $i: \mathbb{H}^n\to \mathbb{B}^n$ so $i^* g_{P}=g_{\mathbb{H}^n}$. 
The isometries of $g_P$ are given by the M\"{o}bius transforms of $\mathbb{B}^n$ and so this identification is not unique.  In fact, for any point $p_0\in \mathbb{H}^n$, there is an isometry $i:  \mathbb{H}^n\to \mathbb{B}^n$ so $i(p_0)=\mathbf{0}$.  Moreover, if $i, j:  \mathbb{H}^n\to \mathbb{B}^n$ satisfy $i(p_0)=j(p_0)=\mathbf{0}$, then $i\circ j^{-1}$ is an orthogonal transformation of $\mathbb{B}^n$.  In particular, in this case $i^*g_{\Real^n}=j^*g_{\Real^n}$ while these metrics are different for identifications associated to distinct distinguished points.  

In the remainder of this article, we will always choose a distinguished point $p_0\in \mathbb{H}^n$ and an identification (i.e., an isometry) $i: \mathbb{H}^n\to \mathbb{B}^n$ with $i(p_0)=\mathbf{0}$. 
We use this identification to compactify $\mathbb{H}^n$ and denote the \emph{ideal boundary} of $\mathbb{H}^n$ by $\partial_\infty \mathbb{H}^n$ which is identified with $\mathbb{S}^{n-1}=\partial \mathbb{B}^n$ by extending $i:\bar{\mathbb{H}}^n\to \bar{\mathbb{B}}^n$ in the obvious way.  This compactification is independent, as a manifold with boundary, of the choice of $p_0$ and $i$.

A complete submanifold $\Sigma\subset \mathbb{H}^n$ has \emph{$C^m$-regular asymptotic boundary} for $1\leq m\leq\infty$ if $\Sigma'=\overline{i(\Sigma)} \subset \bar{\mathbb{B}}^n$ is a $C^m$-regular manifold with boundary, $\partial {\Sigma}'\subset \mathbb{S}^{n-1}=\partial \mathbb{B}^n$ that meets $\mathbb{S}^{n-1}$ orthogonally. Denote by $\partial_\infty \Sigma$ the submanifold corresponding to $\partial {\Sigma}'$ in $\partial_\infty \mathbb{H}^n$.   As M\"{o}bius transformations are smooth and conformal this is a well defined notion independent of choice of identification.

Using the identification, $i$, $\partial_{\infty} \mathbb{H}^n$ has a well defined Riemannian metric induced from $\mathbb{S}^{n-1}=\partial \mathbb{B}^n$.  While this metric depends on $p_0$, it is otherwise independent of the choice of isometry taking $p_0$ to $\mathbf{0}$. Let us denote this metric by $g_{\partial_\infty \mathbb{H}^n}^{p_0}$. 
  Clearly, $g_{\partial_\infty \mathbb{H}^{n}}^{p_0}$ and $g_{\partial_\infty \mathbb{H}^{n}}^{q_0}$ are conformal for different choices of distinguished point $p_0$ and $q_0$ and so $\partial_\infty \mathbb{H}^n$ has a well defined conformal structure. In fact, the two metrics are related by a M\"{o}bius transform on the sphere.   Fix a $l$-dimensional $C^m$ submanifold $\Gamma\subset \partial_\infty \mathbb{H}^n$ and let $i(\Gamma)\subset \mathbb{S}^{n-1}$ be the corresponding submanifold of the sphere under the identification.  Set
$$
Vol_{\partial_\infty \mathbb{H}^n}(\Gamma,p_0)=Vol_{\mathbb{S}^{n-1}}(i(\Gamma))=Vol_{\Real^n}(i(\Gamma)).
$$
If $q_0$ is a different choice of distinguished point, then,  there is a M\"{o}bius transform, $\psi\in \mathrm{Mob}(\mathbb{S}^{n-1})$ so that
$$
Vol_{\partial_\infty \mathbb{H}^n}(\Gamma,q_0)=Vol_{\mathbb{S}^{n-1}}(\psi(i(\Gamma)))
$$
Hence, the \emph{conformal volume}  of $\Gamma\subset \partial_\infty \mathbb{H}^n$ defined by
$$
\lambda_{c}[\Gamma]=\sup_{\psi\in \mathrm{Mob}(\mathbb{S}^{n-1})} Vol_{\mathbb{S}^{n-1}}(\psi(i(\Gamma))),
$$
is well defined independent of the choice distinguished point and of identification.  

In fact, the quantity $\lambda_c[\Gamma]$ is essentially the $n$-conformal volume of the embedding defined by Li-Yau \cite{Li1982}.   Moreover, as the M\"{o}bius transformations of $\mathbb{S}^{n-1}$ are parameterized by $\mathbf{a}\in \mathbb{B}^n$ in an explicit way one has
$$
\lambda_{c}[\Gamma]=\sup_{\mathbf{a}\in \mathbb{B}^n} \int_{\Gamma} \frac{(1-|\mathbf{a}|^2)^{l/2}}{(1-\mathbf{a}\cdot \mathbf{x}(p))^l} dVol_{\Gamma}(p)
$$
 In particular, as shown in  \cite{bryantSurfacesConformalGeometry1988}, this readily leads to the following elementary properties:
\begin{lem}\label{BoundaryConfVolRigidityLem}
	Let $\Sigma\subset \mathbb{H}^{n}$ be a $C^1$-asymptotically regular $l$-dimensional minimal submanifold.  One has
	$$
	\lambda_{c}[\partial_\infty \Sigma]\geq Vol_{\mathbb{R}^n} (\mathbb{S}^{l-1})
	$$
	with equality if and only if $\Sigma$ is a totally geodesic copy of $\mathbb{H}^l$. Moreover, if 
	$$
	\lambda_{c}[\partial_\infty \Sigma]> Vol_{\mathbb{R}^n} (\mathbb{S}^{l-1}),
	$$ 
	then there is a $p_0\in \mathbb{H}^{n}$
	so that
	$$
	\lambda_{c}[\partial_\infty \Sigma]=Vol_{\partial_\infty \mathbb{H}^{n}}(\partial_\infty \Sigma, p_0).
	$$
\end{lem}
\begin{proof}
	As $\partial_\infty \Sigma$ is a closed submanifold of $\mathbb{S}^{n-1}$, the  first claim is an immediate consequence of applying \cite[Proposition 1]{bryantSurfacesConformalGeometry1988} to the embedding $\phi:\partial_\infty \Sigma \to \partial_\infty \mathbb{H}^{n}$. That result also gives that equality holds only when $\partial_\infty \Sigma$ is a totally geodesic copy of $\mathbb{S}^{l-1}$ in $\mathbb{S}^{n-1}$.  The rigidity of $\Sigma$ in case of equality follows immediately from this. The final claim is, likewise, an immediate consequence of \cite[Corollary 1]{bryantSurfacesConformalGeometry1988}. 
\end{proof}

\section{Asymptotic expansion of length and area}

We record here certain computations involving geometric quantities near the boundary of the compactification, $\Sigma'\subset \bar{\mathbb{B}}^n$, of an asymptotically regular minimal surface $\Sigma$ in $\mathbb{H}^n$.

\begin{prop}\label{ComputationLem}
	
	Let $\Sigma $ be  a two-dimensional minimal surface in $\mathbb{H}^n$, $n\geq 3$, with a $C^2$-regular asymptotic boundary. 
	Fix a point $p_0\in \mathbb{H}^n$ and let $\Sigma'$ be the compactification in $\bar{\mathbb{B}}^n$ corresponding to $\Sigma$ and $p_0$.
	
	One has, 
\begin{equation}\label{Claim1Eqn}
	\frac{1}{4}(1-|\mathbf{x}|^2)\mathbf{H}_{\Sigma'}^{g_E}=\mathbf{x}^\perp.
\end{equation}	
	As a consequence, for  $p\in  \partial \Sigma'\subset \partial \mathbb{B}^n$,
\begin{equation}\label{Claim2Eqn}
	\mathbf{x}^\perp(p)=\mathbf{0}
\end{equation}
and
\begin{equation}\label{Claim3Eqn}
	\mathbf{A}_{\Sigma'}^{g_E}(\mathbf{x}, \mathbf{x})|_{p}=\mathbf{k}_{\partial \Sigma'}^{\mathbb{S}^{n-1}}(p).
\end{equation}
	Hence,
\begin{equation}\label{Claim4Eqn}
	\begin{aligned}
	Vol_{\Real^n}(\Sigma'\cap \partial B_s)=s Vol_{\Real^n}(\partial \Sigma')-&\frac{(s-1)^2}{2}\int_{\partial \Sigma'}|\mathbf{k}_{\partial \Sigma'}^{\mathbb{S}^{n-1}}|^2  dVol_{\partial \Sigma'} \\
	&+o((s-1)^2), s\to 1
	\end{aligned}
\end{equation}
	and
\begin{equation}\label{Claim5Eqn}
\int_{\Sigma'\cap \partial B_s} \frac{|\mathbf{x}|}{|\mathbf{x}^\top|} dVol_{\Sigma'\cap \partial B_s}=sVol_{\Real^n}(\partial \Sigma')+o((s-1)^2), s\to 1.
\end{equation}	
\end{prop}
\begin{proof}
	By construction, $\Sigma'$ is minimal with respect to $g_P$ and is $C^2$ up to $\partial \Sigma'$. As
	$$
	g_P=\frac{4}{(1-|\mathbf{x}|^2)^2} g_E,
	$$
	the formula for the transformation mean curvature vector under conformal change of metric implies that, on $\Sigma'\cap \mathbb{B}^n$,
\begin{equation*}\label{ConfHEqn}
	\mathbf{H}^{g_P}_{\Sigma'}=\frac{(1-|\mathbf{x}|^2)^2}{4}\left( \mathbf{H}_{\Sigma'}^{g_E}-\frac{4 \mathbf{x}^\perp}{1-|\mathbf{x}|^2}\right).
\end{equation*}
	As $\Sigma'$ is minimal with respect to $g_P$ this yields \eqref{Claim1Eqn}.
As $\Sigma'$ is $C^2$, it follows that the left hand side of \eqref{Claim1Eqn} vanishes on $\partial \Sigma'$.  This means $\mathbf{x}^\perp|_{\partial \Sigma'}=0$,  proving \eqref{Claim2Eqn}.
Observe that the normal connection to $\Sigma'$ satisfies, for any $\mathbf{v}$ tangent to $\Sigma'$,
$$
\nabla^{\perp, \Sigma'}_{ \mathbf{v}}\mathbf{x}^\perp= -\mathbf{A}_{\Sigma'}^{g_E} (\mathbf{x}^\top, \mathbf{v}).
$$
It follows that for any $p\in \partial \Sigma'$
\begin{equation}\label{Limitxperpeqn}
\lim_{\substack{q\in \Sigma'\setminus \partial \Sigma'\\ q\to p}}\frac{\mathbf{x}^\perp(q)}{1-|\mathbf{x}(q)|}= \mathbf{A}_{\Sigma'}^{g_E} (\mathbf{x}, \mathbf{x})|_{p}.
\end{equation}
This together with \eqref{Claim1Eqn} and the continuity of the mean curvature implies that along $\partial \Sigma'$ 
$$
\frac{1}{2}\mathbf{H}^{g_E}_{\Sigma'}=\mathbf{A}_{\Sigma'}^{g_E} (\mathbf{x}, \mathbf{x}).
$$
Clearly, along $\partial \Sigma'$,
$$
\mathbf{H}^{g_E}_{\Sigma'}=\mathbf{A}_{\Sigma'}^{g_E} (\mathbf{x}, \mathbf{x})+\mathbf{A}_{\Sigma'}^{g_E} (\mathbf{T} , \mathbf{T}),
$$
where $\mathbf{T}$ is a choice of tangent vector along $\partial \Sigma'$ with $\vert \mathbf{T}\vert_{g_E}=1$.  Taken together, this implies that on $\partial \Sigma'$
$$
\mathbf{A}_{\Sigma'}^{g_E} (\mathbf{T} , \mathbf{T})=\mathbf{A}_{\Sigma'}^{g_E} (\mathbf{x}, \mathbf{x}).
$$
Finally, as $\mathbf{x}^\perp=\mathbf{0}$ on $\partial \Sigma'$ one has
$$
\mathbf{k}_{\partial\Sigma'}^{\mathbb{S}^{n-1}}=\mathbf{A}_{\Sigma'}^{g_E} (\mathbf{T} , \mathbf{T})=\mathbf{A}_{\Sigma'}^{g_E} (\mathbf{x}, \mathbf{x}).
$$
This proves \eqref{Claim3Eqn}.

For $t\geq 0$, let
$\phi_t:\Sigma'\to \Sigma'$
be the flow of a $C^1$ vector field $\mathbf{V}$ that satisfies
$$
\mathbf{V}=-\frac{|\mathbf{x}|\mathbf{x}^\top}{|\mathbf{x}^\top|^2}.
$$
near $\partial \Sigma'$.  One verifies that $\mathbf{V}$ has been chosen so that, near $\partial \Sigma'$, $\nabla_{\mathbf{V}}^{g_E} |\mathbf{x}|=-1$ and so, for $t$ sufficiently small, $\phi_t(\partial \Sigma')\subset \partial B_{1-t}$.
Clearly, \eqref{Claim2Eqn} implies that for $p\in \partial \Sigma'$,
$$
\frac{d}{dt}|_{t=0} \mathbf{x}(\phi_t(p))=-\mathbf{x}(p).
$$
One further concludes from \eqref{Claim3Eqn} that
\begin{align*}
\frac{d^2}{dt^2}|_{t=0}\mathbf{x}(\phi_t(p)) & =\frac{d}{dt}|_{t=0}  \mathbf{V}(\phi_t(p))=\nabla_{-\mathbf{x}}^{g_E} \mathbf{V}|_{p} \\
&=\mathbf{A}_{\Sigma'}^{g_E}(\mathbf{x}, \mathbf{x})|_p=\mathbf{k}_{\partial \Sigma}^{\mathbb{S}^{n-1}}(p).
\end{align*}

The fact that $\mathbf{x}^\perp=\mathbf{0}$ on $\partial \Sigma'$ further implies that, along $\partial \Sigma'$,
$$
\mathbf{k}_{\partial \Sigma'}=\mathbf{k}_{\partial\Sigma'}^{\mathbb{S}^{n-1}}-\mathbf{x}.
$$
Hence, using the first variation formula,
\begin{align*}
\frac{d}{ds}|_{s=1} Vol_{\Real^n}(\Sigma'\cap \partial B_s)&=-\frac{d}{dt}|_{t=0} Vol_{\Real^n} (\phi_t(\partial \Sigma')\\
&=
\int_{\partial \Sigma'} \mathbf{k}_{\partial \Sigma'}\cdot (-\mathbf{x})\; dVol_{\partial\Sigma'}=Vol_{\Real^n}(\partial\Sigma').
\end{align*}
Likewise, the second variation formula for length (see for instance \cite[Equation (9.4)]{Simon1983})
yields
\begin{align*}
	\frac{d^2}{ds^2}|_{s=1} &Vol_{\Real^n}(\Sigma'\cap \partial B_{s})=\frac{d^2}{dt^2}|_{t=0} Vol_{\Real^n}(\Sigma'\cap \partial B_{1-t})=\frac{d^2}{dt^2}|_{t=0} Vol_{\Real^n}(\phi_t(\partial \Sigma'))\\
	&=-\int_{\partial \Sigma'} \mathbf{k}_{\partial \Sigma'}\cdot \mathbf{k}_{\partial \Sigma}^{\mathbb{S}^{n-1}} dVol_{\partial \Sigma'} +\int_{\partial \Sigma'} (-\mathbf{k}_{\partial \Sigma'}\cdot -\mathbf{x})^2-(-\mathbf{x}\cdot\mathbf{k}_{\partial \Sigma'})^2 dVol_{\partial \Sigma'}\\
	&= -\int_{\partial \Sigma'}|\mathbf{k}_{\partial \Sigma}^{\mathbb{S}^{n-1}}|^2 dVol_{\partial \Sigma'}.
\end{align*}
Together these calculations prove \eqref{Claim4Eqn}.

Observe that
$$
\frac{|\mathbf{x}|}{|\mathbf{x}^\top|}= \frac{1}{\sqrt{1-\frac{|\mathbf{x}^\perp|^2}{|\mathbf{x}|^2}}}.
$$
As such, at any point in $\Sigma$
\begin{equation*}
\nabla_{\mathbf{x}^\top}^{g_E} \frac{|\mathbf{x}|}{|\mathbf{x}^\top|}= \frac{\frac{\mathbf{x}^\perp}{|\mathbf{x}|^2} \cdot\nabla_{{\mathbf{x}^\top}}^{\perp, \Sigma'} \mathbf{x}^\perp-\frac{|\mathbf{x}^\perp|^2|\mathbf{x}^\top|^2}{|\mathbf{x}|^4}}{\left(1-\frac{|\mathbf{x}^\perp|^2}{|\mathbf{x}|^2}\right)^{3/2}}\\
=-\frac{|\mathbf{x}|\mathbf{x}^\perp \cdot \mathbf{A}_{\Sigma'}^{g_E} (\mathbf{x}^\top, \mathbf{x}^\top)}{|\mathbf{x}^\top|^3}-\frac{|\mathbf{x}^\perp|^2}{|\mathbf{x}| |\mathbf{x}^\top|}.
\end{equation*}
Hence, for $p\in \partial \Sigma'$ and $t\geq 0$ small,
\begin{equation}\label{DerivNormEqn}
	\begin{aligned}
	\frac{d}{dt} \frac{|\mathbf{x}(\phi_t(p))|}{|\mathbf{x}^\top(\phi_t(p))|}&= 
	\left.\nabla_{\mathbf{V}}^{g_E}  \frac{|\mathbf{x}|}{|\mathbf{x}^\top|}\right|_{\phi_t(p)}\\
	&=\left.\left(\frac{|\mathbf{x}|^2\mathbf{x}^\perp \cdot \mathbf{A}_{\Sigma'}^{g_E} (\mathbf{x}^\top, \mathbf{x}^\top)}{|\mathbf{x}^\top|^5}-\frac{|\mathbf{x}^\perp|^2}{ |\mathbf{x}^\top|^3}\right)\right|_{\phi_t(p)}.
	\end{aligned}
\end{equation}
In particular, this vanishes when $t=0$. Combined with the first variation formula for length this implies
$$
	\frac{d}{ds}|_{s=1} \int_{\Sigma'\cap \partial B_s} \frac{|\mathbf{x}|}{|\mathbf{x}^\top|} dVol_{\Sigma'\cap \partial B_s}=Vol_{\Real^n}(\partial \Sigma').
$$
Moreover, for $p\in \partial \Sigma'$ differentiating \eqref{DerivNormEqn} at $t=0$ yields
\begin{align*}
\frac{d^2}{dt^2}|_{t=0} \frac{|\mathbf{x}(\phi_t(p))|}{|\mathbf{x}^\top(\phi_t(p))|}&=\lim_{t\to 0} \frac{1}{t} \left(\frac{|\mathbf{x}|^2\mathbf{x}^\perp \cdot \mathbf{A}_{\Sigma'}^{g_E} (\mathbf{x}^\top, \mathbf{x}^\top)}{|\mathbf{x}^\top|^5}-\frac{|\mathbf{x}^\perp|^2}{ |\mathbf{x}^\top|^3}\right)|_{\phi_t(p)}\\
&=\mathbf{A}_{\Sigma'}^{g_E} (\mathbf{x}^\top, \mathbf{x}^\top)|_{p} \lim_{t\to 0} \frac{\mathbf{x}^\perp(\phi_t(p))}{1-|\mathbf{x}(\phi_t(p)|} = |\mathbf{k}_{\partial \Sigma'}^{\mathbb{S}^{n-1}}|^2(p).
\end{align*}
Where we used \eqref{Claim2Eqn}, \eqref{Limitxperpeqn} and \eqref{Claim3Eqn}.
Combining this with the second variation formula shows 
$$
	\frac{d^2}{ds^2}|_{s=1} \int_{\Sigma'\cap \partial B_s} \frac{|\mathbf{x}|}{|\mathbf{x}^\top|} dVol_{\Sigma'\cap \partial B_s}=0.
$$
The expansion \eqref{Claim5Eqn} follows immediately.
\end{proof}

We also need the following geometric expansions that refine \eqref{VolExpEqn} and \cite[Lemma 4.1]{bernsteinColdingMinicozziEntropy2020} for surfaces with $C^2$-regular asymptotic boundary.  
\begin{prop}\label{LengthAreaAsympProp}
	Let $\Sigma $ be  a two-dimensional minimal surface in $\mathbb{H}^n$, $n\geq 3$, with a $C^2$-regular asymptotic boundary.  For any $p_0\in \mathbb{H}^n$, there are constants $L_\infty,  A_\infty$ and $K_\infty$ so that:

$$
L_R=Vol_{\mathbb{H}^n} (\Sigma\cap \partial B_{R}^{\mathbb{H}^n}(p_0))=L_\infty \sinh R-K_\infty e^{-R} +o(e^{-R}), R\to \infty.
$$
and
$$
A_R=Vol_{\mathbb{H}^n}(\Sigma\cap \bar{B}_R^{\mathbb{H}^n}(p_0))=L_\infty \cosh R +A_\infty+o(e^{-R}), R\to \infty.
$$
Furthermore, $L_\infty=Vol_{\partial_\infty\mathbb{H}^n} (\partial_\infty \Sigma, p_0)$, $A_\infty=\mathcal{A}(\Sigma)$
and
$$
K_\infty=\int_{\partial \Sigma'} |\mathbf{k}_{\partial \Sigma'}^{\mathbb{S}^{n-1}}|^2 dVol_{\partial \Sigma'},
$$
where $\Sigma'\subset \bar{\mathbb{B}}^n$ is the compactification of $\Sigma$ with respect to $p_0$. 
\end{prop}
\begin{proof}
	Pick an identification $i:\mathbb{H}^{n}\to \mathbb{B}^{n}$ with $i(p_0)=\mathbf{0}$.  As $i^{*}g_P=g_{\mathbb{H}^{n}}$, one has
	$i(\partial B_R^{\mathbb{H}^{n}}(p_0))=\partial B_{R}^{g_P}(\mathbf{0})$.  Furthermore, as the conformal factor of $g_P$ is radial, $\partial B_{R}^{g_P}(\mathbf{0})=\partial B_{s}$ where $R=\ln \left( \frac{1+s}{1-s}\right) $, equivalently, $s=\frac{e^R-1}{e^R+1}$.
	
	Set $\Sigma_R=\Sigma\cap \partial B_{R}^{\mathbb{H}^{n}}(p_0)$. 
	and let $\Sigma'$ be the closure of $i(\Sigma)$ be the natural compactification of $\Sigma$ relative to $p_0$.  From the above,   $i(\Sigma_R)=\Sigma'\cap \partial B_s(0)=\Sigma'_s$.  Let $g_R$ be the metric on $\Sigma_R$ induced from $\mathbb{H}^{n}$ and $g_{s}'$ be the metric induced on $\Sigma'_s$ from $g_{\Real^{n}}$.  Clearly,
	$$
	i^*g_{s}'= \left( \frac{e^R-1}{e^R+1}\right)^2\sinh^{-2}(R) g_R=\frac{1}{(1+\cosh R )^2} g_R.
	$$
	In particular, as $\Sigma_s'$ is $1$-dimensional,
	$$
	Vol_{\Real^{n}}(\Sigma_s')= \frac{Vol_{\mathbb{H}^{n}}(\Sigma_R)}{1+\cosh R }.
	$$
	Set
	$$L_\infty=Vol_{\Real^n}(\partial \Sigma')=Vol_{\partial_\infty \mathbb{H}^n}(\partial_\infty \Sigma, p_0).
	$$
	As $\Sigma$ has a $C^2$-regular asymptotic boundary, \eqref{Claim4Eqn} of Proposition \ref{ComputationLem}, ensures
	$$
    Vol_{\Real^{n}}(\Sigma_s')=sL_\infty -\frac{1}{2}(s-1)^2 K_\infty +o((s-1)^2), s\to 1
	$$
	where
	$$
	K_\infty= 		\int_{\partial \Sigma'} |\mathbf{k}_{\partial \Sigma'}^{\mathbb{S}^{n-1}}|^2 dVol_{\partial \Sigma'}.
	$$
	As
	$$
	s=\frac{e^R-1}{e^R+1}=\frac{\sinh R}{1+\cosh R}
	$$
	and
	$$
	1-s=\frac{2}{e^R+1}=2e^{-R}+o(e^{-R}), R\to \infty,
	$$
	we conclude
\begin{align*}
  L_R&=	Vol_{\mathbb{H}^n}(\Sigma\cap \partial{B}_{R}^{\mathbb{H}^n}(p_0))=Vol_{\mathbb{H}^{n}}(\Sigma_R)=(1+\cosh R )Vol_{\Real^n}(\Sigma'_s)\\
 	&=L_\infty \sinh R -K_\infty e^{-R} +o(e^{-R}), R\to \infty.
\end{align*}
This gives the first expansion.	
	

	To work out the area expansions first observe that on $\mathbb{B}^n$ one has
	$$
	(\nabla_{\Sigma'}^{g_P} \rho)(p)= \frac{1-|\mathbf{x}(p)|^2}{2|\mathbf{x}(p)|} \mathbf{x}^\top(p)
	$$
	where $\rho(p)=\dist_{g_P}(p,\mathbf{0})$.
	
	It follows that
	$$
	\frac{1}{|\nabla_{\Sigma'}^{g_P} \rho|_{g_P}(p)}=\frac{|\mathbf{x}(p)|}{|\mathbf{x}^\top(p)|}.
	$$
	Hence, as $\Sigma$ has $C^2$-regular asymptotic boundary, \eqref{Claim5Eqn} of Proposition \ref{ComputationLem} implies
	$$
	\int_{\Sigma_s'} \frac{1}{|\nabla_{\Sigma'}^{g_P} \rho|_{g_P}}dVol_{\Sigma_s'}= L_\infty s+o((s-1)^2),s\to 1.
	$$ 
	Using the identifications from before yields
\begin{align*}
	\int_{\Sigma_R} \frac{1}{|\nabla_{\Sigma} r|} dVol_{\Sigma_R}&= (1+\cosh R) \int_{\Sigma_s'} 	\frac{1}{|\nabla_{\Sigma'}^{g_P} \rho|_{g_P}}dVol_{\Sigma'_s}\\
	&= L_\infty \sinh R+o(e^{-R}), R\to \infty
\end{align*}
where $r(p)=\dist_{\mathbb{H}^n}(p,p_0)$. 
 The co-area formula ensures,
	$$
	\frac{d}{dR} Vol_{\mathbb{H}^n}(\Sigma\cap \bar{B}_R^{\mathbb{H}^n}(p_0))=	\int_{\Sigma_R} \frac{1}{|\nabla_{\Sigma} r|} dVol_{\Sigma_R}.
	$$
and so
	$$
	\frac{d}{dR} Vol_{\mathbb{H}^n}(\Sigma\cap \bar{B}_R^{\mathbb{H}^n}(p_0))= L_\infty \sinh R+o(e^{-R}), R\to \infty.
	$$
	Hence,
	\begin{align*}
	\frac{d}{dR} \left( Vol_{\mathbb{H}^n}(\Sigma\cap \bar{B}_R^{\mathbb{H}^n}(p_0)) -L_\infty \cosh R\right) =o(e^{-R}), R\to \infty.
	\end{align*}
	As the error is integrable, we deduce
	\begin{align*}
	A_R=Vol_{\mathbb{H}^n}&(\Sigma\cap \bar{B}_R^{\mathbb{H}^n}(p_0)) = L_\infty \cosh R+A_\infty+o(e^{-R})
	\end{align*}
	Where 
\begin{align*}
	A_\infty &=	\lim_{R\to \infty}\left( Vol_{\mathbb{H}^n}(\Sigma\cap \bar{B}_R^{\mathbb{H}^n}(p_0)) -L_\infty \cosh R\right)\\
	&=\lim_{R\to \infty}\left( Vol_{\mathbb{H}^n}(\Sigma\cap \bar{B}_R^{\mathbb{H}^n}(p_0)) -Vol_{\mathbb{H}^n}(\Sigma\cap \partial B_R^{\mathbb{H}^n}(p_0))\right).
\end{align*}
	The definition of renormalized area ensures $A_{\infty}=\mathcal{A}(\Sigma)$ is independent of $p_0$ -- see Lemma \ref{AMFormualLem}.
\end{proof}

\section{Proof of the main theorem}

In order to prove the main result we need two auxilliary results inspired by \cite{Choe1992}.
First of all, given a point $p_0\in \mathbb{H}^n$ and a curve $\gamma\subset \partial B_{R}^{\mathbb{H}^n}(p_0)$, define the \emph{cone of $\gamma$ over $p_0$} to be
$$
C(\gamma; p_0)=\set{p\in \mathbb{H}^n: p\in \sigma(p_0,q) \mbox{ for some } q\in \gamma}.
$$
Here $\sigma(p_0,q)$ is the minimzing geodesic segment connecting $p_0$ to $q$.
If $r=\dist_{\mathbb{H}^n}(\cdot, p_0)$, one readily checks that the vector field $\nabla_{\mathbb{H}^n} r$ is tangent to $C=C(\gamma; p_0)\setminus\set{p_0}$,

When $\gamma\subset \partial B_{R}^{\mathbb{H}^n}(p_0)$, one has the following simple formula relating length and area of geodesic balls in $C(\gamma; p_0)$ centered at $p_0$.
\begin{lem}\label{ConeComputationLem}
Suppose $\gamma\subset \partial B_{R}^{\mathbb{H}^n}(p_0)$ is a $C^2$ curve.  For any $0<\rho\leq R$, let
	$$
	 L_\rho^C=Vol_{\mathbb{H}^n}(C(\gamma; p_0)\cap \partial B_{\rho}^{\mathbb{H}^n}(p_0))
	$$
	and
	 $$
	 A_\rho^C=Vol_{\mathbb{H}^n}(C(\gamma; p_0)\cap  \bar{B}_{\rho}^{\mathbb{H}^n}(p_0))
	 $$
	The density at $p_0$ of the cone satisfies, for all $0<\rho\leq R$,
	 $$
	 \Theta= \Theta(C(\gamma; p_0), p_0)=\frac{L_\rho^C}{2\pi \sinh \rho}=\frac{A^C_\rho}{2\pi (\cosh \rho-1)}.
	 $$
	 Moreover, one has
	 $$
	 (L_\rho^C)^2=4\pi \Theta A_\rho^C +(A_\rho^C)^2=\frac{2 L_\rho^C A_\rho^C}{\sinh \rho} +(A_\rho^C)^2.
	 $$
\end{lem}
\begin{proof}
For any $C^2$ curve $\sigma\subset \partial B_{\rho}^{\mathbb{H}^n}(p_0)$ one computes
$$
\mathbf{k}_{\sigma}=\mathbf{k}_\sigma^\top -\coth \rho  \nabla_{\mathbb{H}^n} r
$$
where $\mathbf{k}_\sigma^\top$ is the component of the curvature tangent to $\partial B_{\rho}^{\mathbb{H}^n}(p_0)$ and $r$ is the radial distance in $\mathbb{H}^n$ to $p_0$.  For $t\leq 0$ let
$$
\phi_t:\mathbb{H}^n\setminus \bar{B}_{-t}^{\mathbb{H}^n}(p_0) \to\mathbb{H}^n\setminus \set{p_0}
$$ 
be the flow of $\nabla_{\mathbb{H}^n} r$.  By definition, this flow preserves cones based at $p_0$.  Moreover, one has
$$
\phi_{-t}(\gamma)=C(\gamma, p_0)\cap B_{R-t}^{\mathbb{H}^n}(p_0).
$$
It follows from the first variation formula that 
$$
\frac{d}{dt} Vol_{\mathbb{H}^n}(\phi_{-t}(\gamma))=-\coth (R-t)  Vol_{\mathbb{H}^n}(\phi_{-t}(\gamma))
$$
Solving this ODE gives for $t\in [0, R)$
$$
 Vol_{\Real^n}(C(\gamma; p_0)\cap B_{R-t}^{\mathbb{H}^n}(p_0)) =Vol_{\mathbb{H}^n}(\phi_{-t}(\gamma))=\sinh(R-t) Vol_{\mathbb{H}^n} (\gamma).
 $$
 That is, for $\rho\in (0,R]$
 $$
L_\rho^C=Vol_{\Real^n}(C(\gamma; p_0)\cap B_{\rho}^{\mathbb{H}^n}(p_0)) = Vol_{\mathbb{H}^n} (\gamma) \sinh \rho.
$$ 
As $\nabla_{\mathbb{H}^n} r$ has unit length away from $p_0$ and is tangent to $C(\gamma; p_0)$ it follows from the co-area formula that for $\rho\in (0, R]$
\begin{align*}
A_\rho^C=Vol_{\mathbb{H}^n}(C(\gamma; p_0)\cap \bar{B}_{\rho}^{\mathbb{H}^n}(p_0)) &=\int_0^\rho Vol_{\mathbb{H}^n}(C(\gamma; p_0)\cap B_{t}^{\mathbb{H}^n}(p_0)) dt\\
&=\int_0^\rho Vol_{\mathbb{H}^n} (\gamma)  \sinh t dt\\
&=(\cosh \rho-1) Vol_{\mathbb{H}^n} (\gamma) \\
&=\frac{\cosh \rho-1}{\sinh \rho} L_\rho^C.
\end{align*}
Since,
$$
\lim_{\rho\to 0}\frac{\cosh \rho-1}{ \rho^2}=\frac{1}{2}
$$
it follows that
$$
\Theta=  \lim_{\rho\to 0}\frac{A_\rho^C}{ \pi \rho^2}=\frac{Vol_{\mathbb{H}^n} (\gamma)}{2\pi}=\frac{L_\rho^C}{2\pi \sinh \rho}=\frac{A_\rho^C}{2\pi(\cosh \rho-1)}. 
$$
This proves the first claim.
To see the second we observe that
$$
(L_\rho^C)^2=(Vol_{\mathbb{H}^n} (\gamma))^2 \sinh^2 \rho 
$$
while
\begin{align*}
4\pi\Theta A_\rho^C+(A_\rho^C)^2&= 2(Vol_{\mathbb{H}^n} (\gamma))^2 (\cosh \rho-1)+ (Vol_{\mathbb{H}^n} (\gamma))^2 (\cosh \rho-1)^2\\
&= (Vol_{\mathbb{H}^n} (\gamma))^2(\cosh^2 \rho-1)= (Vol_{\mathbb{H}^n} (\gamma))^2\sinh^2 \rho.
\end{align*}
This proves the second claim.
\end{proof} 

\begin{prop}\label{AreaLengthEst}
Suppose $\Sigma$ is a compact minimal surface in $\mathbb{H}^n$ with $\partial \Sigma \subset \partial B_{R}^{\mathbb{H}^n}(p_0)$ a $C^1$ curve. 
If
$$
L_R= Vol_{\mathbb{H}^n}(\partial \Sigma)=Vol_{\mathbb{H}^n}(\Sigma\cap \partial B_{R}^{\mathbb{H}^n}(p_0)) 
$$
and
$$
A_R=Vol_{\mathbb{H}^n}(\Sigma\cap \bar{B}_{R}^{\mathbb{H}^n}(p_0)),
$$
then
$$
L_R^2\geq \frac{2L_R}{\sinh R} A_R +A_R^2.
$$
\end{prop}
\begin{proof}
 The key fact, proved in \cite[Proposition 2]{Choe1992} is the fact that any minimal surface in $\mathbb{H}^n$ has less area than an appropriate cone competitor.  That is,  if $\Sigma\subset \mathbb{H}^n$ is a compact minimal $p_0\in \mathbb{H}^n \setminus \partial \Sigma$, then
 $$
Vol_{\mathbb{H}^n}(C(\partial \Sigma; p_0))\geq  Vol_{\mathbb{H}^n}(\Sigma).
 $$
 For the sake of completeness we provide a proof this fact in Lemma \ref{AreaCompLem}.
 
 Combining this estimate with Lemma \ref{ConeComputationLem} for $\Sigma$ and $p_0$ as in the hypothesis gives
 $$
 L_R^2 = (L_R^C)^2 =\frac{2L_R^C A_R^C}{\sinh R} +(A_R^C)^2\geq \frac{2 L_R^C A_R}{\sinh R} + A_R^2 = \frac{2 L_R A_R}{\sinh R} +A_R^2.
 $$
 This completes the proof.
\end{proof}

We can now prove Theorem \ref{MainThm}:
\begin{proof}
Set
$$
L_R= Vol_{\mathbb{H}^n}(\Sigma\cap \partial B_{R}^{\mathbb{H}^n}(p_0)) 
$$
and
$$
A_R=Vol_{\mathbb{H}^n}(\Sigma\cap \bar{B}_{R}^{\mathbb{H}^n}(p_0))
$$
 As $\Sigma$ has $C^2$-regular asymptotic boundary, the expansions of Proposition \ref{LengthAreaAsympProp} imply
\begin{align*}
	L_R^2 &=L_\infty^2 \sinh^2 R-K_\infty +o(1), R\to \infty,
\end{align*}
\begin{align*}
	\frac{2L_R A_R}{\sinh R}& =2L_\infty^2 \cosh R+2 L_\infty A_\infty+o(1), R\to \infty
\end{align*}
and 
\begin{align*}
	A_R^2&=L_\infty^2 \cosh^2 R+2L_\infty A_\infty \cosh R+A_\infty^2 +o(1), R\to \infty.
\end{align*}

%
%
Plugging these expansions into the inequality given by Proposition \ref{AreaLengthEst} yields
\begin{align*}
L_\infty^2 \sinh^2 R &\geq 2L_\infty^2\cosh R +L_\infty^2\cosh^2 R\\
&+2L_\infty A_\infty \cosh R +O(1), R\to \infty.
\end{align*}
Rearranging this inequality and using $\cosh^2 R=\sinh^2 R+1$ gives
$$
-2L_\infty^2 \cosh R\geq 2L_\infty A_\infty \cosh R +O(1), R\to \infty.
$$

As $Vol_{\partial_\infty \mathbb{H}^n}(\partial_\infty \Sigma, p_0)=L_\infty>0$ and $A_\infty =\mathcal{A}(\Sigma)$ this yields
$$
-Vol_{\partial_\infty \mathbb{H}^n}(\partial_\infty \Sigma, p_0)\geq   \mathcal{A}(\Sigma) +O(e^{-R}), R\to \infty
$$
Taking the limit as $R\to \infty$ implies
$$
-Vol_{\partial_\infty \mathbb{H}^n}(\partial_\infty \Sigma, p_0)\geq  \mathcal{A}(\Sigma).
$$
Hence, by taking the suprememum over all $p_0\in \mathbb{H}^n$, one obtains 
$$
-\lambda_{c}[\partial_\infty \Sigma]\geq  \mathcal{A}(\Sigma).
$$
As, Lemma \ref{BoundaryConfVolRigidityLem} implies $\lambda_c[\partial_\infty \Sigma]\geq 2\pi$ all the claimed estimates hold.  The first rigidity result follows from the fact that if one has equality throughout, then $\lambda_c[\Sigma]=2\pi$ and so by Lemma \ref{BoundaryConfVolRigidityLem} implies $\Sigma$ is a totally geodesics copy of $\mathbb{H}^2$.

To complete the proof, we suppose $\lambda_c[\partial_\infty \Sigma]=-\mathcal{A}(\Sigma)$.
In this case, by Lemma \ref{BoundaryConfVolRigidityLem} either $\lambda_c[\partial_\infty\Sigma]=2\pi$ and $\Sigma$ is a totally geodesic $\mathbb{H}^2$ or $\lambda_c[\partial_\infty\Sigma]>2\pi$ and there is a $p_0\in \mathbb{H}^n$ so
$$
\lambda_c[\partial_\infty \Sigma]=Vol_{\partial \infty \mathbb{H}^n}(\partial_\infty \Sigma; p_0). 
$$
By Proposition \ref{LengthAreaAsympProp}, this means
$$
\lambda_c[\partial_\infty \Sigma]=Vol_{\partial \infty \mathbb{H}^n}(\partial_\infty \Sigma; p_0)=\lim_{R\to \infty} Vol_{\mathbb{H}^n} (\Sigma \cap \partial B_{R}^{\mathbb{H}^n}(p_0))=\lim_{R\to \infty} L_R=L_\infty.
$$
For this $p_0$, the rigidity hypothesis gives
$$
L_\infty=\lambda_c[\partial_\infty \Sigma] = -\mathcal{A}(\Sigma)=-A_\infty.
$$
Hence we may rewrite two of the expansions above as
\begin{align*}
\frac{2L_R A_R}{\sinh R}&=2A_\infty^2\cosh R-2 A_\infty^2+o(1), R\to \infty
\end{align*}
and, using $\cosh^2 R=\sinh^2 R+1$, 
\begin{align*}
A_R^2&=L_\infty^2 \sinh^2 R- 2A_\infty^2\cosh R+2A_\infty^2+o(1), R\to \infty  
\end{align*}
Combining these with the estimate of Proposition \ref{AreaLengthEst} and canceling terms gives
$$
-K_\infty\geq o(1), R\to \infty.
$$
Sending $R\to \infty$ gives,
$$
0\geq K_\infty= \int_{\partial \Sigma'} |\mathbf{k}_{\partial \Sigma'}^{\mathbb{S}^{n-1}}|^2 dVol_{\partial \Sigma'} \geq 0.
$$
Hence,
$$
 \int_{\partial \Sigma'} |\mathbf{k}_{\partial \Sigma'}^{\mathbb{S}^{n-1}}|^2 dVol_{\partial \Sigma'}=0
 $$
and so $\partial \Sigma'\subset\partial \mathbb{B}^n$ is a closed geodesic in $\mathbb{S}^{n-1}=\partial \mathbb{B}^n$. As $\partial \Sigma'$ has multiplicity one, this means
$$
2\pi =Vol_{\Real^n}(\partial \Sigma')=Vol_{\partial_\infty \mathbb{H}^n}(\partial_\infty \Sigma, p_0)=\lambda_{c}[\partial_\infty \Sigma]=-\mathcal{A}(\Sigma).
$$
This returns us to the previous rigidity situation and so $\Sigma$ must be a totally geodesic $\mathbb{H}^2$.
\end{proof}
\appendix
\section{Area properties of minimal surfaces in $\mathbb{H}^n$}
For the sake of completeness we include proofs of two facts we use in this paper about the area of minimal surfaces in $\mathbb{H}^n$.  The first is an area comparison result from \cite{Choe1992}:
\begin{lem}\label{AreaCompLem}
Let $\Sigma$ be a compact minimal surface with boundary in $\mathbb{H}^n$.  For any $p_0\not\in \partial \Sigma$ one has
$$
Vol_{\mathbb{H}^n} (\Sigma)\leq Vol_{\mathbb{H}^n}(C(\partial \Sigma; p_0)).
$$
\end{lem}
\begin{proof}
 For $p\neq p_0$ consider the vector field on $\mathbb{H}^n$
 $$
 \mathbf{X}(p)=\frac{\cosh r(p)-1}{\sinh r(p)} \nabla_{\mathbb{H}^n} r=\frac{\cosh r(p)-1}{\sinh^2 r(p)} \nabla_{\mathbb{H}^n} \cosh r
 $$
 where 
 $$
 r(p)=\dist_{\mathbb{H}^n}(p,p_0).
 $$
 The vector field $\mathbf{X}$ extends smoothly to $p=p_0$.  Using
 $$
 \nabla^2_{\mathbb{H}^n} \cosh r = \cosh r \;  g_{\mathbb{H}^n},
 $$
 we see that for any surface $\Gamma\subset \mathbb{H}^n$
 $$
 \mathrm{div}_{\Gamma} \mathbf{X}=1+\frac{(\cosh r(p)-1)^2+1}{\sinh^2 r(p)} (1-|\nabla_\Gamma r|^2)
 $$ 
 In particular, if $\Gamma$ is minimal, then
 $$
 \mathrm{div}_{\Gamma} \mathbf{X}^\top=\mathrm{div}_{\Gamma} \mathbf{X} \geq  1.
 $$
 Likewise, if $\Gamma$ is a cone over $p_0$, then $\nabla_{\mathbb{H}^n} r$ is tangent to $\Gamma$ and so $|\nabla_\Gamma r|=|\nabla_{\mathbb{H}^n} r|=1$.  Hence, as $\mathbf{X}$ is also tangent to $\Gamma$,
 $$
  \mathrm{div}_{\Gamma} \mathbf{X}^\top=g_{\mathbb{H}^n}(\mathbf{H}_\Gamma, \mathbf{X})+\mathrm{div}_{\Gamma} \mathbf{X}=1.
 $$
 
If $\nu$ is the outward normal to $\Sigma$ along $\partial \Sigma$ and $\eta$ is the outward normal to $C(\partial \Sigma; p_0)$, then, along $\partial \Sigma$,
 $$
 g_{\mathbb{H}^n}(\nu, \mathbf{X})\leq g_{\mathbb{H}^n}(\eta, \mathbf{X})=1.
 $$
 Hence, the first variation formula gives
 \begin{align*}
 Vol_{\mathbb{H}^n}(\Sigma)& = \int_{\Sigma} 1 dVol_{\Sigma}\leq \int_{\Sigma} \mathrm{div}_{\Sigma} (\mathbf{X}) d{Vol}_{\Sigma}\\
                           &=\int_{\partial \Sigma}  g_{\mathbb{H}^n}(\nu, \mathbf{X})
                           dVol_{\partial \Sigma}\leq \int_{\partial \Sigma}  g_{\mathbb{H}^n}(\eta, \mathbf{X})dVol_{\partial \Sigma}\\
                           &=\int_{\partial \Sigma}\mathrm{div}_{C(\partial \Sigma,p_0)} (\mathbf{X}) dVol_{C(\partial \Sigma,p_0)}=\int_{\partial \Sigma}1 dVol_{C(\partial \Sigma,p_0)}\\
&=Vol_{\mathbb{H}^n}(C(\partial \Sigma; p_0)).
 \end{align*}
\end{proof}

We also show that the renormalized area is defined independent of point $p_0$ and also verify formula \eqref{AMFormulaEqn} from \cite{Alexakis2010}:
\begin{lem}\label{AMFormualLem}
Let $\Sigma $ be  a two-dimensional minimal surface in $\mathbb{H}^n$, $n\geq 3$, with a $C^2$-regular asymptotic boundary.  For any $p_0\in \mathbb{H}^n$,	one has
$$
\mathcal{A}(\Sigma)=\lim_{R\to \infty}\left( Vol_{\mathbb{H}^n}(\Sigma\cap \bar{B}_R^{\mathbb{H}^n}(p_0))- Vol_{\mathbb{H}^n}(\Sigma\cap \partial {B}_R^{\mathbb{H}^n}(p_0))\right)
$$
is independent of $p_0$.  Moreover, 
$$
\mathcal{A}(\Sigma)=-2\pi \chi(\Sigma)-\frac{1}{2}\int_{\Sigma} |\mathbf{A}^{\mathbb{H}^n}_\Sigma|^2 dVol_{\Sigma}.
$$
\end{lem}
\begin{proof}
	Pick an identification $i:\mathbb{H}^{n}\to \mathbb{B}^{n}$ so that $i(p_0)=\mathbf{0}$.  As $i^{*}g_P=g_{\mathbb{H}^{n}}$, one has
$i(\partial B_R^{\mathbb{H}^{n}}(p_0))=\partial B_{R}^{g_P}(\mathbf{0})$. Set $\Sigma_R=\Sigma\cap \partial B_{R}^{\mathbb{H}^{n}}(p_0)$. 
and let $\Sigma'=\overline{i(\Sigma)}$ be the natural compactification of $\Sigma$ relative to $p_0$.  Clearly,   $i(\Sigma_R)=\Sigma'\cap \partial B_s(0)=\Sigma'_s$.  

Let $\mathbf{k}_{\Sigma_R}$ be the mean curvature vector of $\Sigma_R$ in $\mathbb{H}^n$.  Likewise, let $\mathbf{k}_{\Sigma_s'}^{g_P}$ be the mean curvature vector of $\Sigma'_s$ with respect to $g_P$ and  $\mathbf{k}_{\Sigma_s'}^{g_E}$ be the mean curvature vector of $\Sigma'_s$ with respect to $g_E$.  The formula for the conformal change of mean curvature vector ensures
$$
\mathbf{k}_{\Sigma_s'}^{g_P}=\frac{(1-|\mathbf{x}|^2)^2}{4}\left( \mathbf{k}_{\Sigma_s'}^{g_E}-\frac{2\mathbf{x}}{1-|\mathbf{x}|^2}\right).
$$
Let $\rho(p)=\dist_{g_P}(p,\mathbf{0})$ and set
$$
\mathbf{N}(p)=\frac{\nabla_{g_P} \rho(p)}{|\nabla_{g_P} \rho|_{g_P}(p)}=\frac{1-|\mathbf{x}(p)|^2}{2|\mathbf{x}^\top(p)|} \mathbf{x}^\top(p).
$$
For $s$ near $1$ this is the outward unit normal to $\Sigma_s'$ in $\Sigma_s$.
On $\Sigma_s'$ one computes,
\begin{align*}
	g_P\left(\mathbf{N},\mathbf{k}_{\Sigma_s'}^{g_P} \right) &= -|\mathbf{x}^\top|+\frac{1-|\mathbf{x}|^2}{2|\mathbf{x}^\top|} g_E(\mathbf{k}_{\Sigma_s'}^{g_E}, \mathbf{x}^\top).
\end{align*}
Observe that for $p\in \partial \Sigma'$ it follows from \eqref{Limitxperpeqn} that 
$$
\lim_{\substack{q\in \Sigma'\setminus \partial \Sigma'\\ q\to p}} \frac{|\mathbf{x}(q)|-|\mathbf{x}^\top(q)|}{1-|\mathbf{x}(q)|}=\lim_{\substack{q\in \Sigma'\setminus \partial \Sigma'\\ q\to p}} \frac{|\mathbf{x}^\perp(q)|^2}{(|\mathbf{x}(q)|+|\mathbf{x}^\top(q)|) (1-|\mathbf{x}(q)|)}=0.
$$
Hence, on $\Sigma_s'$,
$$
|\mathbf{x}^\top|=s+o(1-s), s\to 1.
$$
Likewise, using \eqref{Claim2Eqn} of Proposition \ref{ComputationLem} and fact that $\Sigma'$ is $C^2$
$$
\lim_{s\to 1} g_E(\mathbf{k}_{\Sigma_s'}^{g_E}, \mathbf{x}^\top)=g_E(\mathbf{k}_{\partial \Sigma'}^{g_E}, \mathbf{x})=-1
$$
and so on $\Sigma'_s$
$$
\frac{1-|\mathbf{x}|^2}{2|\mathbf{x}^\top|} g_E(\mathbf{k}_{\Sigma_s'}^{g_E}, \mathbf{x}^\top)=s-1+o(s-1), s\to 1.
$$
That is,  on $\Sigma'_s$,
$$
	g_P\left(\mathbf{N},\mathbf{k}_{\Sigma_s'}^{g_P} \right) =-1+o(s-1), s\to 1
$$
We conclude that for $R$ large the geodesic curvature of $\Sigma_R$ in $\Sigma$ satisfies
$$
\kappa_{\Sigma_R}=-g_{\mathbb{H}^n}\left( \mathbf{k}_{\Sigma_R}, \frac{\nabla_{\Sigma} r}{|\nabla_\Sigma r|}\right)=1+o(e^{-R}), R\to \infty
$$
where here $r(p)=\dist_{\mathbb{H}^n}(p,p_0)$.  Hence,
$$
\int_{\Sigma_R} \kappa_{\Sigma_R} dVol_{\Sigma_R}=Vol_{\mathbb{H}^n} (\Sigma_R)+o(1), R\to \infty
$$
where we used the fact proved in see \cite[Lemma 4.1]{bernsteinColdingMinicozziEntropy2020} or the first part of Proposition \ref{LengthAreaAsympProp} that
$$
Vol_{\mathbb{H}^n} (\Sigma_R)=Vol_{\partial_\infty \mathbb{H}^n}(\partial_\infty \Sigma, p_0) \sinh R+o(e^R), R\to \infty.
$$

Finally, by the Gauss equations, if $K_\Sigma$ is the Gauss curvature, then, as $\Sigma$ is minimal
$$
K_{\Sigma}=-1-\frac{1}{2}|\mathbf{A}_{\Sigma}|^2
$$
Hence, by the Gauss-Bonnet formula, for $R$ large,
\begin{align*}
Vol_{\mathbb{H}^n}&(\Sigma\cap \bar{B}_R^{\mathbb{H}^n}(p_0) )=-\int_{\Sigma\cap \bar{B}_R^{\mathbb{H}^n}(p_0)} K_{\Sigma} +\frac{1}{2}|\mathbf{A}_{\Sigma}|^2 dVol_{\Sigma}\\
&=\int_{\Sigma_R} \kappa_{\Sigma_R} dVol_{\Sigma_R}-2\pi \chi(\Sigma\cap \bar{B}_R^{\mathbb{H}^n}(p_0)) -\frac{1}{2}\int_{\Sigma\cap \bar{B}_R^{\mathbb{H}^n}(p_0)}|\mathbf{A}_{\Sigma}|^2 dVol_{\Sigma}\\
&=Vol_{\mathbb{H}^n} (\Sigma_R) -2\pi \chi(\Sigma)-\frac{1}{2}\int_{\Sigma\cap \bar{B}_R^{\mathbb{H}^n}(p_0)}|\mathbf{A}_{\Sigma}|^2 dVol_{\Sigma}+o(1), R\to \infty.
\end{align*}
Here the last equality used that the definition of asymptotically regular boundary means the Euler characteristic of $\Sigma\cap \bar{B}_R^{\mathbb{H}^n}(p_0)$ stabilizes for large $R$. Both claims follow immediately from this by sending $R\to \infty$. Note that the existence of the limit, but not it's independence from $p_0$, follows from the first part of Proposition \ref{LengthAreaAsympProp} and the proof of this part does not use Lemma \ref{AMFormualLem}.  
\end{proof}
\section{Graham Witten expansion}\label{GrahawWittenApp}

To connect our expansion of \eqref{VolExpEqn} with that considered by Graham and Witten \cite{RobinGraham1999} introduce the function 
$$
s=2\frac{1-|\mathbf{x}|}{1+|\mathbf{x}|}
$$
on $\mathbb{B}^{n+k} \backslash\set{0}$. One verifies that $s$ is a boundary defining function and,
in appropriate associated  coordinates,  the Poincar\'{e} metric has the form
$$
g_P= s^{-2} \left( ds^2+ \left(1-\frac{s^2}{4}\right)^{2} g_{\mathbb{S}^{n+k-1}}\right).
$$
Hence, if $i$ is an identification of $\mathbb{H}^{n+k}$ with $\mathbb{B}^{n+k}$ sending $p_0$ to $0$ and $\sigma=s\circ i$, then $\sigma^2 g_{\mathbb{H}^{n+k}}$ is a conformal compactification in the sense of \cite{RobinGraham1999}.
Suppose $\Sigma\subset \mathbb{H}^{n+k}$ is a $n$-dimensional minimal submanifold that has $C^\infty$-regular asymptotic boundary. Using a boundary defining function like $\sigma$,  Graham and Witten \cite{RobinGraham1999} showed that, when $n$ is even,
\begin{align*}
Vol_{\mathbb{H}^{n+k}}(&\Sigma\cap B_{R(\epsilon)}^{\mathbb{H}^{n+k}}(p_0))=Vol_{\mathbb{H}^{n+k}}(\Sigma\cap \set{\sigma>\epsilon})\\
&=c_0 \epsilon^{-n+1}+c_2 \epsilon^{-n+3}+\cdots+c_{n-2} \epsilon^{-1} +c_{n-1} +o(1), \epsilon\to 0
\end{align*}
while if $n$ is odd, then
\begin{align*}
Vol_{\mathbb{H}^{n+k}}(&\Sigma\cap  B_{R(\epsilon)}^{\mathbb{H}^{n+k}}(p_0))=
Vol_{\mathbb{H}^{n+k}}(\Sigma\cap \set{\sigma>\epsilon})\\
&= c_0 \epsilon^{-n+1}+c_2 \epsilon^{-n+3}+\cdots+c_{n-3} \epsilon^{-2} +d \log \frac{1}{\epsilon} +c_{n-1} +o(1), \epsilon\to 0.
\end{align*}
Here $R(\epsilon)=-\ln\left(\frac{\epsilon}{2}\right)$.  When $n=2$, if one sets
$$
\epsilon(R)=2e^{-R}
$$
then the the expansion is equivalent to \eqref{VolExpEqn} as in this case
$$
\cosh R(\epsilon)=\epsilon^{-1}+o(1), \epsilon \to 0.
$$

Graham and Witten  further showed that, when $n$ is even, $c_{n-1}$ is independent of the choice of $\sigma$ (and so independent of choice of $p_0$)  and the same is true of $d$ when $n$ is odd.
When $n=2$ this quantity is precisely the renormalized area considered in \cite{Alexakis2010} completing the justification of \eqref{VolExpEqn}. This expansion is also the (Riemannian) analog of the entropy considered by Ryu and Takayanagi \cite{Ryu2006, Ryu2006a}. One observes that
$$
c_0=Vol_{\partial_{\infty}\mathbb{H}^{n+k}}(\partial_\infty \Sigma,p_0)
$$
so the conformal volume is related to the leading order behavior of this expansion.  Comparing with \cite{BWRelEnt} and \cite{BWTopUniq} it seems the renormalized area is analogous to the relative entropy.

\bibliographystyle{hamsabbrv}
\bibliography{Library}
\end{document}